%% file: nse-inf-dim.tex
\documentclass[a4paper,10pt]{amsart}
\usepackage[T1]{fontenc}
\usepackage[utf8]{inputenc}
\usepackage{epsfig}
\usepackage{amssymb,amsmath,amsthm,mathrsfs,mathtools,amscd}
\usepackage{float}
\usepackage[figtopcap]{subfigure}
\usepackage{pgfplots}
\usepackage{todonotes}
\usepackage[bookmarks=true]{hyperref}

\input{def}

\input{mathEnv}

\title[Robust Stabilization of Oseen Systems]{Convergence of Coprime Factor Perturbations for Robust Stabilization of Oseen Systems} % under Linearization Uncertainties}
\author{Jan Heiland}

\begin{document}
\maketitle
% \tableofcontents
% \input{mathEnv}
\begin{abstract}
    Linearization based controllers for incompressible flows have been proven to work in theory and in simulations. To realize such a controller numerically, the infinite dimensional system has to be linearized and discretized. The unavoidable consistency errors add a small but critical uncertainty to the controller model which will likely make it fail, especially when an observer is involved. Standard robust controller designs can compensate small uncertainties if they can be qualified as a coprime factor perturbation of the plant. We show that for the linearized Navier-Stokes equations, a linearization error can be expressed as a coprime factor perturbation and that this perturbation smoothly depends on the size of the linearization error. In particular, improving the linearization makes the perturbation smaller so that, eventually, standard robust controller will stabilize the system.
\end{abstract}

\section{Introduction}

Linearization based compensator design for stabilizing a steady state of the Navier-Stokes equations has been thoroughly investigated in theory, see, e.g. \cite{Ray06} for early fundamental results and \cite{HeHZ18} for a recent extension to observer design, and proved to work in numerical simulations \cite{BenH15}. In view of applications that come with various sources of unmodelled effects, however, standard observer-based controllers are likely to fail because of their inherent lack of robustness \cite{Doy78} even in the finite-dimensional case.

For this paper we consider the situation that the infinite dimensional model can be stabilized via a linearization based controller. In order to realize a stabilizing control in a simulation, two approximation steps are imminent: the discretization of the continuous states both of the plant and the controller and a computation of the linearization. Both approximations can be arbitrarily close but will always introduce an error in the model used for the controller design. In this work we focus on the linearization error. We show that this error can be qualified via perturbations in coprime factorizations of the transfer functions which can be compensated for via \Hinf~controller design; see \cite{McFG90} for the finite dimensional and \cite{CurZ95} for the infinite dimensional case.

For that we consider a simulation setup (Section \ref{sec-problem-setup}) for the control of the \emph{flow around a cylinder}, which is a widely used benchmark for flow simulations \cite{SchT96}. We relax the modelled \emph{Dirichlet} control conditions towards \emph{Robin}-type conditions which allows for a direct integration in the variational formulations. We show that in this particular setup, which, however, is readily extendable to other flow control problems, the associated \emph{Oseen} operator generates an analytic semigroup and that the input operator is bounded (Section \ref{sec-oseen-lti}). 
The boundedness is then used to infer the existence of a transfer function in frequency domain whereas the analyticity provides additional regularity of the solutions in time domain. For an output operator that makes the system uniformly detectable, we then show that a perturbation to the \emph{Oseen} linearization amounts to a perturbation of coprime factors of the transfer function (Section \ref{sec-transfer-cpf}). We characterize those perturbations and, by means of the analyticity and a state-space realization, show that the coprime factor perturbations tend to zero in the relevant norm as the linearization error approaches zero. 

The implication of this general result is that standard $\Hinf$-robust output based controllers that are computed on the base of an inexact linearization will eventually stabilize the nominal system. 

The \emph{Oseen} linearization as a base for controller design for the incompressible Navier-Stokes equations has been thoroughly analysed as, e.g., in \cite{Ray06,Ray07a,NguR15} for Dirichlet boundary control in two and three dimensions and with mixed boundary conditions, in \cite{Bad06} in two and three dimensions including space discretizations and Riccati-based state feedback, in \cite{HeHZ18} in view of observer design, in \cite{Rav07} with extension methods for the Dirichlet control, and in the textbook \cite{Bar11} with fundamental results on spectral properties. As for robust control, the \emph{Oseen} linearization has been investigated in \cite{DhaRT11, Bar11} in view of the operator $\Hinf$-Riccati equations. Separately, the need and applicability for $\Hinf$-robust controllers to account for discretization or linearization errors in flow control setups have been discussed in \cite{BenH16,BenH17}.

For general distributed parameter systems, the design and advantages $\Hinf$-robust controllers have been thoroughly treated in \cite{CurZ95} and analysed with respect to discretization and model reduction errors \cite{Cur03,Cur06,PauP19}. Those results were directed to systems with \emph{bounded} input and output operators. For more general system class, the question of robust stabilizability and stable factorizations of transfer functions was treated in \cite{Cur90}. Finally, we mention the most general comprehensive notion for infinite-dimensional systems, namely \emph{well-posed} or \emph{regular} linear systems as treated in the fundamental work \cite{Sal87}, in the text books \cite{Sta05,TucW09}, and in the recent survey paper \cite{TucW14} that also covers some nonlinear systems and the Navier-Stokes equations. 

With this paper we contribute two aspects to the existing theory: First, we review general results on regularity and controller design for the \emph{Oseen} system and adapt them to the practically relevant setup with \emph{Robin}-type boundary conditions and inexactness in the linearization. Second, derive a state space realization of the coprime factor perturbations to the transfer function that are caused by inexact linearizations. We then use the regularity to prove that these perturbations scale with the inexactness in the linearization in the relevant norm. While this smooth dependency is readily established in finite dimensional state space systems, in infinite dimensions, this has not been addressed yet.

\section{Problem Setup}\label{sec-problem-setup}

\begin{figure}
	\input{tikz/cylsetup}
	\caption{Computational domain of the cylinder wake.}
	\label{fig:cyldom}
\end{figure}
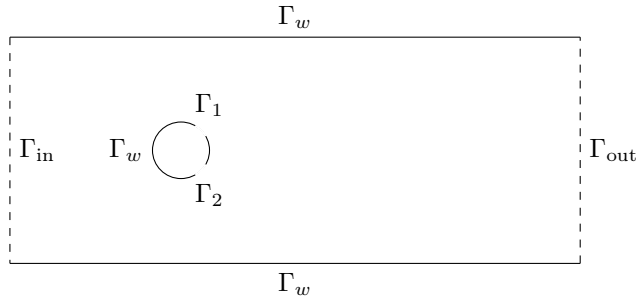
Before we define the semigroup setting, we introduce the considered Navier-Stokes Equations in strong form for an example application. We only consider the two-dimensional case and comment on the extension to three dimensions whenever the results do not simply carry over.

\subsection{Derivation of the System Equations}
We consider the flow of an incompressible fluid in a domain $\Omega\subset \mathbb R^{2}$ with a boundary $\Gamma = \gai \cup \gaw \cup \gao \cup \gaco \cup \gact$, where \gai~stands for the inflow boundary, \gaw~stands for physical walls, \gao~denotes the outflow boundary, and \gaco~and \gact~denote two boundaries where controls can be applied. An example is the cylinder wake with two boundary controls at the cylinder periphery as illustrated in Figure \ref{fig:cyldom}. 

For time $t>0$, we describe the system via the evolution of the velocity $\infv$ and the pressure $\infp$ of the flow via the Navier-Stokes equations

\begin{subequations}\label{eq-ininse}
	\begin{align}
        \dot \infv +(\infv \cdot \nabla) \infv+ \dive \sigma(\infv, \infp)  &=  0, \\
		\dive \infv &= 0 ,\quad \text{in } \Omega,
	\end{align}
\end{subequations}
with the stress tensor 
\begin{equation*}
    \sigma (\infv, \infp) := \nu (\nabla \infv + {\nabla \infv}^\trp ) - \infp I,
\end{equation*}
where the parameter $\nu>0$ is the \emph{dynamic viscosity}\footnote{We assume that the density of the fluid is constant and equal to one}.

As for the inflow and outflow boundary conditions, we apply
\begin{equation*}
    \infv = -ng_0  \text{ on }\gai \quad\text{and}\quad \sigma(\infv, \infp)n =0 \text{ on }\gao,
\end{equation*}
where $g_0$ models the spatial shape of the inflow profile and where $n$ denotes the outward normal vector to the boundary. As for the controls, we prescribe
\begin{equation}\label{eq-nse-diri-cont-bcs}
    \infv = - ng_1\cdot u_1  \text{ on }\gaco \quad\text{and}\quad \infv =-ng_2\cdot u_2 \text{ on }\gact.
\end{equation}
Here, the control input is modelled via the shape functions $g_1$ and $g_2$ and via scalar functions of time $u_1$ and $u_2$ as the control parameters. 

At the walls we assume \emph{no-slip} conditions, that is 
\begin{equation*}
    \infv = 0\text{ on }\gaw.
\end{equation*}
% By means of $\alpha$, describing the magnitude of the velocity at the inflow, we will parametrize the Reynolds number and, thus, the varying flow regimes.

In a system-theoretic formulation, the control of partial differential equations through Dirichlet boundary conditions as \eqref{eq-nse-diri-cont-bcs} does not fit the general setup of $(A,B,C)$ systems with bounded $B$ and $C$; see \cite[Ch. 3.3]{CurZ95}. For such models the notions of \emph{regularity} and  \emph{well-posedness} have been developed, see, e.g., \cite{GuoZ05} for an overview and a positive example. We avoid these difficulty by relaxing the present Dirichlet boundary \eqref{eq-nse-diri-cont-bcs} conditions to Robin-type conditions:
\begin{equation}\label{eq-nse-rob-cont-bcs}
	% v = -ng_i u_i \quad \rightarrow \quad  
    % \infv = -ng_i u_i - \gamma (\nu \partieln \infv - \infp n) \text{ on }\Gamma_i, \quad i=1,2,
    \infv = -ng_i u_i - \gamma \sigma ( \infv , \infp )n \text{ on }\Gamma_i, \quad i=1,2,
\end{equation}
with a parameter $\gamma$ that is supposed to be small, cf., e.g., \cite{HouR98} for convergence properties of this relaxation in optimal control of stationary flows.

Next we assume that the corresponding steady state equations with zero control input have a solution $(\vstst, \pstst)$ and we consider the difference states

\[
    \vd := \vstst - \infv \quad\text{and}\quad \pd := \pstst - \infp
\]
subject to the linearized dynamics

\begin{subequations}\label{eq-lin-ininse}
	\begin{align}
        \dot \vd +(\vd \cdot \nabla) \vstst+(\vstst \cdot \nabla) \vd + \dive \sigma(\vd, \pd)  &=  0, \\
		\dive \vd &= 0 ,\quad \text{in } \Omega,\\
    \intertext{with boundary conditions}
         \sigma( \vd , \pd )n +  \frac{1}{\gamma}  \vd &= - \frac{1}{\gamma} ng_i\cdot u_i  \text{ on }\gaci \quad\text{for } i=1,2 \\
    \vd &= 0 \text{ on }\gai\cup \gaw \\
    \sigma(\vd, \pd)n&=0 \text{ on }\gao. \label{eq-lin-ininse-robinbc}
	\end{align}
\end{subequations}

% \subsection{Parametrization of the State Equations}

% \subsection{Abstract System Formulations}

System \eqref{eq-lin-ininse} is commonly referred to as \emph{Oseen linearization} and provides a suitable model for controller design for incompressible flows around the working point; see, e.g. \cite{Son98} for a general discussion of the \emph{linearization principle} and \cite{Ray06,Ray07b} for analytical results confirming that the linear controller will locally stabilize the associated nonlinear Navier-Stokes equations.

\subsection{Weak Formulation}

By testing with a test function $w\in L^2(\Omega)$ that is sufficiently smooth, divergence free, and zero at the Dirichlet boundaries, one derives the standard weak formulation 
\begin{equation}\label{eq-osee-weak}
    \ltp{\dot \dv}{w} + a_\gamma(\dv, w) + b_\infty (\dv, w) = -\frac{1}{\gamma}(\int_{\gaco} g_1 u_1 w \inva s + \int_{\gact} g_2 u_2 w \inva s)
\end{equation}
from system \eqref{eq-lin-ininse}, with
\begin{subequations}\label{eq-lin-ac-forms}
    \begin{align}
        a_\gamma(v, w):&= \int_\Omega \nu (\nabla v +\nabla v^\trp):(\nabla w + \nabla w^\trp) \inva x + \frac{1}{\gamma} \int _{\gac} vw \inva s,\\
        b_\infty (v, w):&= \int_\Omega ((\vstst \cdot \nabla)v)\cdot w + ((v\cdot \nabla)\vstst)\cdot w \inva x.
    \end{align}
\end{subequations}

\subsection{Notation and Basic Definitions}

For  $\Omega \subset \mathbb R^{2}$ denoting a domain, we consider the \emph{Sobolev} spaces $\Ltod$~-- the space of $\mathbb R^{2}$ valued functions that are square integrable over $\Omega$ -- and $H^s( \Omega;\mathbb R^{2} )$ -- the subspace of $\Ltod$ of functions for which also all derivatives up to order $s$ are contained in \Ltod~(in the case that $s$ is a positive integer). 
For positive $s\in\mathbb R\setminus \mathbb N^{}$, the space $H^s( \Omega;\mathbb R^{2} )$ can be defined consistently; see \cite{Rou05} for an introduction into \emph{Sobolev} spaces and the definitions of the corresponding norms. 
Here, we will use subscripts to depict the various norms whereas $\ltp{\cdot}{\cdot}$ will denote the $\Ltod$ inner product. 

As for the treatment of the flow equations, we introduce subspaces of divergence free functions that are zero at the Dirichlet boundaries
\begin{align*}
	\VI &:= \{z \in \Hood: \dive z = 0 \text{ and } z\bigr | _{\gaw\cup\gai} = 0 \}, \\
	\VO &:= \{z \in \Ltod: \dive z = 0 \text{ and } z\cdot n\bigr | _{\gaw\cup\gai} = 0 \};
\end{align*}
see, e.g. \cite{Ray06}, which are closed subspaces of $\Hood$~and \Ltod~in the corresponding norms. Note that by the identity mapping $\VI$  is continuously and densely embedded in $\VO$. We use the notation $\VI\hookrightarrow \VO$ to stress on this topological inclusion as opposed to the algebraical subset.
% and the dual space $\Vdivd$ with respect to the dense embedding $\Vdivg \hookrightarrow \Hdivn$, cf. \cite{NguR15}. 
Furthermore, \VO~is a closed subspace of the Hilbert space \Ltod, which, among others, means that there exists an orthogonal projector 
\begin{equation*}
	\Pi\colon \Ltod \text{ onto } \VO \subset \Ltod,
\end{equation*}
and that we can identify $\VO$ with its dual space $ {\VO}^*$. 

% In view of expressing solutions to abstract linear systems, we will make use of the \emph{Bochner} space
% \begin{equation*}
% W(0,T;\VI;\VO) := \{v\in L^2( 0, T; \VI)\colon \dot v \in L^2( 0, T; \VO) \},
% \end{equation*}
% for end times $T\in (0, \infty]$. Throughout this work, we use the time derivative in the classical sense.

For two Banach spaces $V_1$ and $V_2$, we denote the space of bounded linear operators from $V_1$ to $V_2$ by $\mathcal L(V_1, V_2)$. Also we will refer to the notion of \emph{strongly continuous semigroup} or $C_0$-semigroup as defined, e.g., in \cite[Def. 2.1.2]{CurZ95}, and to an operator $A$ being the \emph{generator of an analytic  $C_0$-semigroup} where \emph{generator} is defined, e.g., in \cite[Def. 2.1.8]{CurZ95} and \emph{analyticity of a $C_0$-semigroup}, e.g, in \cite[Def. II-1.2.1]{BenDDM07}.

Throughout the manuscript, a lower case $c$ with, possibly, a subscript will denote a positive constant.

For generic linear time-invariant systems, we will use the shortcut $(A,B,C)$ to refer to a state space realization. To connect it to a transfer function, we will write $G\sim(A,B,C)$. We will consider stable transfer functions $G\in\mhamz$ that, by definition, are also contained in $\mathcal M\hat{\mathcal A}(0)$ and are thus holomorphic and bounded on the right half plane $\Czp$; see \cite[Def. 7.1.4, Lem. A.7.47]{CurZ95}. Accordingly $G\in\mhamz$ can be measured in terms of the \Hinf-norm
\begin{equation*}
    \|G\|_{\Hinf} = \sup_{s\in \Czp}\|G(s)\|_{\mathcal L(U, Y)},
\end{equation*}
where $U$ and $Y$ are the input and the output space.
% In view of defining solutions and their regularity to the flow equations we make pose the following standing assumption on the domain $\Omega$: 
% 
% \begin{ass}
%     The domain $\Omega $ is such that
%     \begin{itemize}
%         \item $\Omega$ is of smooth class and the junctions of Dirichlet/Dirichlet or Dirichlet/Neumann are as follows...
%     \end{itemize}
% 
% \end{ass}
% Note that this assumption is fulfilled by the example setup depicted in Figure \ref{fig:cyldom}.

% We define the dualities as extensions of the \Ltod~inner product and identify the pivot spaces $\Ltod$ and $\Hdivn$ and their duals. We will make use of this by tacitly identifying forms and vectors in $\Ltod$ and $\Hdivn$. 

\section{Oseen as Linear System}\label{sec-oseen-lti}

For the linearization point that defines the \emph{Oseen} system \eqref{eq-osee-weak} we make the following regularity assumption.
\begin{ass}\label{ass-vstst-hthe}
    The weak solution $\vstst$ to the steady-state with no control action and their approximation $\vstst+\dv$  exist and fulfill 
    \begin{equation*}
        \vstst,\;\vstst+\dv \in \Hthe,
    \end{equation*}
    for some $\varepsilon > 0$.
\end{ass}

This assumption is in line with the assumptions made in \cite[(H5)]{NguR15}. Establishing the validity of this assumption for given setups is a delicate task. In the appendix of \cite{NguR15}, the authors provide reasoning for the existence of $\vstst$ in a similar setup like the one considered here (see Figure \ref{fig:cyldom}). Once existence is confirmed, regularity can be derived from the general results presented in \cite{MazR09}.

Informally, we introduce the \emph{Oseen} operator $A$ as
\begin{equation*}
    A\colon D(A)\subset \VO \to \VO \colon v\mapsto \Pi(\nu \Delta v - (\vstst \cdot \nabla)v - (v\cdot \nabla)\vstst)
\end{equation*}
where 
% \begin{equation}
%     D(A) := H^2(\Omega;\mathbb R^{2} ) \cap \VI.
% \end{equation}
    \begin{equation}\label{eq-oseen-dom-def}
        D(A)=\{ v \in H^{3/2+\varepsilon_0}(\Omega; \mathbb R^2): \exists p \in H^{3/2+\varepsilon_0}(\Omega;\mathbb R^{} ) \text{ s.t. } \dive \sigma(v, p) \in L^2( \Omega, \mathbb R^{2} ) \}
    \end{equation}
    for some $\varepsilon_0>0$.

\begin{rem}
    Because of corners in the domain and because of junctions of Dirichlet and Neumann-type conditions, solutions to the related Stokes problem, i.e. the case where $\vstst=0$, might be in $H^{3/2+\varepsilon_0}(\Omega; \mathbb R^2)$, $\varepsilon_0 \in (0, 1/2)$, rather than in $H^2$. Accordingly, this domain of definition has been proposed for a similar setup \cite{NguR15}; cp. also Assumption \ref{ass-vstst-hthe}. Also note that only the existence of such a $p$ is needed without further specification, since the contribution of $p$ in $\sigma(v,p)$ is in the kernel of $\Pi$. % Since we will not make use of regularity beyond \VI, for our general considerations we stick with the definition \eqref{eq-oseen-dom-def} for the domain of definition.
\end{rem}

Formally, we define $A$ in weak form via 
\begin{equation}\label{eq-weak-oseen-op}
    \ltp{-Av}{w} = a_\gamma(v,w) + b_\infty (v, w) \quad \text{for all }w\in \VI;
\end{equation}
% where the boundary term comes from the boundary conditions; 
with the forms $a_\gamma$ and $b_\infty$ as defined in \eqref{eq-lin-ac-forms}. Since the projector $\Pi$ is self-adjoint in the $L^2( \Omega; \mathbb R^{2} )$-inner product and since $\VI\subset \VO$ the projection of $A$ is contained in the definition \eqref{eq-def-tb} by virtue of 
\begin{equation*}
    \ltp{-Av}{w} = \ltp{-Av}{\Pi w}= \ltp{-\Pi ^* Av}{w} = \ltp{-\Pi  Av}{w} \quad \text{for all } w\in\VI. 
\end{equation*}
Also note the sign of $A$ that is chosen in view of writing the \emph{Oseen} system as $\dot v= Av+Bu$.

\begin{thm}
    If $\vstst$ fulfills Assumption \ref{ass-vstst-hthe}, then for $\gamma>1/\nu$, the Oseen operator $A$ with Robin-type boundary conditions as defined in \eqref{eq-weak-oseen-op} with domain of definition $D(A)$ as defined in \eqref{eq-oseen-dom-def} generates an analytic $C_0$-semigroup on $\VO$.
\end{thm}
\begin{proof}
    We show that $A$ meets the assumptions of \cite[Ch. 2-I, Thm. 2.12.]{BenDDM07}. For that, we recall that $\VI\hookrightarrow  \VO$ is a continuous embedding and that $D(A)$ is chosen such that $\ltp{-Az}{\cdot}$ defines a continuous functional on $\VO$. Since $\gac$ is of nonzero measure, the \emph{Friedrich's inequality} (see, e.g, \cite[Thm. 1.8]{Nec12}) combined with \emph{Korn's inequality} (\cite[p. 193]{Nec12}) that states the ellipticity for the symmetrized gradient  $\mathcal E_s(v):=\nabla v + \nabla v^T$ gives that the \emph{Stokes} part $a_\gamma$ of $A$ is \VI~coercive:
    \begin{align*}
        a_\gamma(z, z)= &\int_\Omega \nu \mathcal E_s( z) :\mathcal E_s(z) \inva x + \frac{1}{\gamma} \int _{\gac} zz \inva s \\
                        & \geq  \nu \bigl( \int_\Omega \mathcal E_s( z) :\mathcal E_s( z )\inva x + \int _{\gac} zz \inva s\bigr) \\
                        &\geq c_a \|z\|_{\VI}^2,
    \end{align*}
    for all $z\in \VI$ and a constant $c_a$ independent of $z$.  By standard arguments, see e.g. \cite[Thm. 2.8]{NguR15}, it then follows that the bilinear form $a_\gamma + b_\infty$ is $\VI$-$\VO$-coercive so that $A$ as defined in \eqref{eq-weak-oseen-op} is indeed a generator of an analytic $C_0$-semigroup.
\end{proof}

\begin{rem}
    The direct estimate of the form $b_\infty$ that is used in \cite[Thm. 2.8]{NguR15} is specific to two dimensions. In the three dimensional case one may add the reasonable assumption that $\vstst \in L^\infty(\Omega, \mathbb R^{3 } )$ to directly obtain the $\VI$-$\VO$-coercivity of $A$ or establish that $A=A_a+A_b$, where $A_b$ is that part of $A$ that is defined via $b_\infty$, is a closed operator with $D(A_a^{1/2})$ as domain of definition and show analyticity with the help of \cite[Ch. 3, Cor. 2.4]{BenDDM07}.
\end{rem}

For sufficiently smooth shape functions $g_1$ and $g_2$, namely $g_i\in H^{1/2}(\Gamma_i;\mathbb R^{})$, $i=1,2$, we define the input operator $B=\Pi \tilde B$ where $B$ is defined as
\begin{equation}\label{eq-def-tb}
    \tilde Bu(w) = -\frac{1}{\gamma} \sum_{i=1,2} \int_{\gaci}g_iw \inva s u_i
\end{equation}
for the two dimensional input $u=(u_1, u_2)$. Note that, due to the evaluation of the trace, $\tilde B$ is \emph{not bounded} in \Ltod~whereas $B=\Pi \tilde B$ is; see \cite{BenH16}.

As for an output operator we consider $C\in \mathcal L(\VO, \mathbb R^{k}) $ that can be, e.g., $k$ averaged velocity measurements over small subdomains of $\Omega$. By similar arguments that ensured $B\in \mathcal L( \mathbb R^2, \VO)$ also boundary observation could be modelled with operators in $\mathcal L(\VO, \mathbb R^{k})$. Finally, as we will show below, the analyticity of the generated semigroup provides additional regularity so that also output operators that are only bounded from \VI~will not necessarily destroy the well posedness of the system.

By the preceding arguments we can state that the linear system 
\begin{subequations}\label{eq-oseen-as-lti}
    \begin{align}
        \dot v(t) &= Av(t) + Bu(t), \quad v(0) \in \VO \\
        y(t) &= Cv(t)
    \end{align}
\end{subequations}
with $A\colon D(A) \subset \VO \to \VO$ being the generator of an analytic $C_0$-semigroup, with $B\in \mathcal L(\mathbb R^{2}, \VO)$ models the evolution of the difference state due to boundary input $u$  in the linear approximation \eqref{eq-lin-ininse} and is observed via the output operator $C\in \mathcal L(\VO, \mathbb R^{k} )$. 
\begin{rem}
    Due to the boundedness of $B$ and $C$ with respect to the state space $\VO$, system \eqref{eq-oseen-as-lti} is of the form that is treated in \cite{CurZ95}. In order to show smooth dependencies of factorizations of the associated transfer function on system perturbations, we need to extend our considerations to a more general class, where additional state regularity may compensate a certain \emph{unboundedness} of $C$.
\end{rem}

\section{Transfer functions and coprime factorizations}\label{sec-transfer-cpf}

For a system $(A, B, C)$ with bounded input and output operators and with $A$ generating a $C_0$-semigroup, the associated \emph{transfer function}
\begin{equation*}
    G(s) = C(sI-A)^{-1}B \in\mathcal L(U,Y)
\end{equation*}
where $U$ and $Y$ denote the input and output space, respectively,
is well defined for all $s\in \rho_\infty(A)$, where $\rho_\infty(A)$ is the maximal part of the resolvent set of $A$ that contains an interval $[r,\infty)$; see \cite[Lem. 4.3.6]{CurZ95}. 

% In what follows 

\begin{thm}[\cite{CurZ95}, Thm. 7.3.8]
    Let $(A,B,C)$ be a linear time-invariant system with state-space $Z$, with input space $\mathbb C^m$ and output space $\mathbb C^k$, and with bounded input operator $B\in\mathcal L(\mathbb C^m,Z)$ and output operator  $C\in\mathcal L(Z, \mathbb C^k)$. If $(A,B,C)$ is exponentially detectable by $L \in \mathcal L(\mathbb C^k, X)$ then 
    % with $N, M \in \Hinf$ and
    \begin{subequations}\label{eq-curz-cpf-formula}
    \begin{align}
        N(s) &= C(sI-A-LC)^{-1}B \\
        M(s) &= I + C(sI-A-LC)^{-1}L
    \end{align}
\end{subequations}
define a (left) coprime factorization $G=M^{-1}N$ of $G$ over $\mathcal M\mathcal A_-(0)$.
\end{thm}

\begin{rem}
    The class of transfer functions that possess coprime factorizations over $\mhamz$ (see \cite[Def. 7.2.7]{CurZ95}) is called the \emph{Callier-Desoer} class and is well suited for infinite dimensional linear time-invariant systems  since, in particular, it contains transfer functions that are not rational as they occur, e.g., in the modelling of systems with delay.
\end{rem}

A perturbation $\dA$ of $A$, that is still stabilized by $LC$, transfers to the coprime factors as follows; cp. \cite[Thm. 4]{BenH16}.

\begin{thm}\label{thm-cpf-prtbtns}
    Let $(A,B,C)$ and $(A+\dA,B,C)$ be linear time-invariant systems with state space $Z$, input space $\mathbb C^m$, and output space $\mathbb C^k$ and with $B\in\mathcal L(\mathbb C^{m},Z)$ and $C\in \mathcal L(Z, \mathbb C^{k})$. If $(A,B,C)$ and $(A+\dA,B,C)$ are simultaneously exponentially detectable by $L \in \mathcal L(\mathbb C^k, X)$ then 
    for $G \sim (A,B,C)$ and $\dG \sim (A+\dA, B, C)$ it holds that 
    \begin{equation*}
        G=M^{-1}N\quad\text{and}\quad \dG = [M+\dM]^{-1}[N+\dN]
    \end{equation*}
    with $N, M, \dN, \dM \in \mhamz$ and
    \begin{subequations}\label{eq-dndm}
    \begin{align}
        \dN(s) = C\dA(sI-A-LC)^{-1}(sI-A-\dA-LC)^{-1}B,\\
        \dM(s) = C\dA(sI-A-LC)^{-1}(sI-A-\dA-LC)^{-1}L.
    \end{align}
\end{subequations}
\end{thm}

\begin{proof}
    Applying formula \eqref{eq-curz-cpf-formula} for the coprime factors to the perturbed system, we obtain 
    \begin{align*}
        N(s)+\dN(s) &= C(sI-A-\dA-LC)^{-1}B ,\\
        M(s)+\dM(s) &= I + C(sI-A-\dA-LC)^{-1}L,
    \end{align*}
    so that the formulas \eqref{eq-dndm} for the differences follow from the operator identity
    \begin{equation*}
        (sI-A-\dA-LC)^{-1}= (sI-A-LC)^{-1}+ \dA(sI-A-LC)^{-1}(sI-A-\dA-LC)^{-1}
    \end{equation*}
    that holds under the assumption that $LC$ stabilizes both $A$ and $A+\dA$.
\end{proof}

The task is now to show that the size of $\dN, \dM$ smoothly depends on the size of the perturbation $\dA$ as this implies that for sufficiently accurate linearizations standard $\Hinf$ robust controller will stabilize the system. In fact, in the finite dimensional case, the needed robustness is proportional to $\| [\dN~\dM]\|_{\Hinf}$; see \cite{BenHW19}.

Since $A$ and $\dA$ are not bounded in general, we cannot use closeness of $A$ and $\dA$ directly to estimate $\dN$. We will rather use a time-domain realization of $\dN$ and $\dM$ to show that, under regularity conditions that we will establish for the Oseen system, the norms of associated input to output or state maps go to zero as the Oseen linearization error approaches zero.

We will make use of the following fundamental result
\begin{thm}[Thm. 1.3, \cite{Wei91}]\label{thm-weiss-transferfunc}
    Let $U$ and $Y$ be Banach spaces and let $\mathcal F$ be an shift-invariant bounded operator from $L^p([0, \infty), U)$ to $L^p([0, \infty), Y)$. Then there exists a unique bounded analytic $\mathcal L(U, Y)$-valued function $H$ defined on $\Czp$ such that for $u$ and $y$ with $y=\mathcal Fu$ it holds that $\hat y(s) = H(s)\hat u(s)$ on $\Czp$ and
    $$
    \sup_{s\in \Czp} \norm{H(s)}_{\mathcal L(U, Y)} \leq \norm{\mathcal F},
    $$
    where $\hat u$ and $\hat y$ are the \emph{Laplace}-transformed signals $u$ and $y$.
\end{thm}

With the help of Theorem \ref{thm-weiss-transferfunc}, we will establish existence of transferfunctions and estimates on their \Hinf-norm on the base of the input-to-state or input-to-output map $\mathcal F$ in the time domain.

\section{Regularity of the Oseen-Semigroup}
The \emph{Oseen operator} has been studied in similar contexts in \cite{Ray06, Bad06, NguR15}. We briefly review the basic analysis for the presented case of \emph{Robin-type} boundary conditions that can be modelled with an input of \emph{distributed type}, meaning that $Bu$ is bounded in the state space $\VO$. Further, we provide results on regularity of solutions as we will need it for the definition of transfer functions and their perturbations.

For the \emph{Oseen }system \eqref{eq-oseen-as-lti}, we have the following regularity result:

\begin{lem}\label{lem-oseen-reg-fintim}
    If 
    %$v_0 \in [D(A), \VO]_{1/2}$ and 
    $u \in L^2(0,\infty; \mathbb R^{2} )$ then for the solution $v$ to the \emph{Oseen} system \eqref{eq-oseen-as-lti} with $v(0)=0$,
    % \begin{equation*}
    %     \dot v = Av + Bu, \quad v(0)=0,
    % \end{equation*}
    it holds that for all $T>0$ 
    \begin{equation}\label{eq-oseensol-finite-time}
        \norm{v}_{L^2(0,T; \VI)} \leq 
        %c_T\norm{v_0}_{[]} + 
        c_T\norm{u}_{L^2(0,T; \mathbb R^{2})} .
    \end{equation}
\end{lem}
\begin{proof}
    Since $B \in \mathcal L(\mathbb R^{2} ; \VO)$ we have that $Bu \in L^2(0,\infty; \VO)$ so that the regularity estimate follows from \cite[Cor. 2.4.7]{Bad06} or \cite[Thm. II-2.2.]{BenDDM07} noting that the general results treat $D(A)$ as a topological subspace of the pivot space and noting that $D(A)\subset \VI$ algebraically and topologically.
\end{proof}

% \begin{itemize}
%     \item  $f\in H$ then $v\in V$
%     \item finite time: input state is bounded in V
%     \item stabilized, infinite time: input to state bounded in H
%     \item stability estimate as in Badra Cor. 2.4.7
%     \item $\dA \colon V \to Y$ goes to zero as \dv goes to zero
% \end{itemize}

Since $c_T$ in \eqref{eq-oseensol-finite-time} may depend on $T$, that estimate does not guarantee the bounded input-to-state map. However, for stable systems the estimates also hold for $T=\infty$. % as mentioned in Badra, if the semigroup generator is of \emph{negative type}, e.g. the generated $C_0$-semigroup is exponentially stable, then one can set $T=\infty$ in the above estimates:

\begin{cor}\label{cor-bounded-itos}
    Consider $A$, $B$, and $C$ from the \emph{Oseen} system \eqref{eq-oseen-as-lti} and let $L\in \mathcal L(\mathbb R^{k} ,\VO)$ such that $A+LC$ is exponentially stable. Then for
    %$v_0 \in [D(A), H]_{1/2}$, 
    $u \in L^2(0,\infty; \mathbb R^{2} )$, the solution $v$ to 
    \begin{equation*}
        \dot v = (A+LC)v + Bu, \quad v(0)=0,
    \end{equation*}
    obeys
    \begin{equation*}
        \norm{v}_{L^2(0,\infty; \VI)} \leq 
        %C\norm{v_0}_{[]} + 
        c_\infty \norm{u}_{L^2(0,\infty; \mathbb R^{2})} .
    \end{equation*}
\end{cor}

\begin{proof}
    Since $C$ and $L$ are bounded $LC\in \mathcal L(\VO,\VO)$ is a bounded perturbation of $A$, so that $A+LC$ still generates an analytic $C_0$-semigroup; see \cite[Ch.3, Thm. 2.2]{Paz83}. Since $A+LC$ is stable, it is \emph{of negative} type, so that $T$ in \ref{eq-oseensol-finite-time} can be set to $T=\infty$, see \cite[Cor. 2.4.7]{Bad06}.
\end{proof}

\begin{rem}
    The estimates of Corollary \ref{cor-bounded-itos} and Lemma \ref{lem-oseen-reg-fintim} can also accomodate an initial condition $v(0)\ne 0$ of sufficient regularity. In view of the estimates of the transfer function, however, only the $v(0)=0$ case is of relevance. 
\end{rem}

% \section{Existence of Uniformly Stabilizing State Feedback}
% \begin{itemize}
%     \item State feedback is robust
%     \item Our numerical results can confirm this
%     \item Maybe the results in Badra also contain this..?
% \end{itemize}

\section{Perturbed linearization and robust stabilization}

In the case of a perturbation of the linearization point $\vstst\approx \vstst + \dv$  the corresponding Oseen operator \eqref{eq-weak-oseen-op} will be perturbed as $A\approx A + \dA$ where $\dA$ is defined via
\begin{equation}
    \ltp{\dA v}{w} = \int_\Omega (\dv \cdot \nabla)v\cdot w + (v\cdot \nabla)\dv\cdot w \inva x  \quad \text{for all }w\in \VI,
\end{equation}
with a domain of definition that contains $D(A)$. 

\begin{lem}\label{lem-da-to-zero}
    If $\dv\in \Hthe$ , then $\dA\in \mathcal L(\VI, \VO)$ and
    \begin{equation*}
        \norm{\dA}_{\mathcal L(\VI, \VO)}\to 0
    \end{equation*}
    as $\norm{\dv}_{\Hthe}\to 0$. 
\end{lem}

\begin{proof}
    We will make use of the estimate 
    \begin{equation}\label{eq-esti-c}
        |\int_\Omega (z \cdot \nabla)v\cdot w \inva x| \le c_b 
        \|z\|_{H^{s_1}(\Omega;\mathbb R^{2})}  
        \|v\|_{H^{1+s_2}(\Omega;\mathbb R^{2})}
        \|w\|_{H^{s_3}(\Omega;\mathbb R^{2})}
    \end{equation}
    that holds for $s_1+s_2+s_3 > 1$; see \cite[Lem. 2.1]{Tem95}.

With the identification of $\VO={\VO}^*$ we have that
\begin{align*}
    \|\dA\|_{\mathcal L(\VI, \VO)} =& \sup_{w\in \VI, \|w\|=1} \|\dA w \|_{\VO}= \sup_{w\in \VI, \|w\|=1} \|\dA w \|_{{\VO}^*} \\
                                    &\leq \sup_{w\in \VI, \|w\|=1}\sup_{v\in \VO, \|v\|=1} |({\dA w}, {v})_{\VO}|\\
                                    &\leq \sup_{w\in \VI, \|w\|=1}\sup_{v\in \VO, \|v\|=1 }|\int_\Omega (\dv \cdot \nabla)v\cdot w + (v\cdot \nabla)\dv\cdot w \inva x|
\end{align*}
which by virtue of \eqref{eq-esti-c} and $(s_1, s_2, s_3)=(\frac{3}{2}+\varepsilon_0, 0, 0)$ for the first and $(s_1, s_2, s_3)=(1, \frac{1}{2}+\varepsilon_0, 0)$ for the second summand can be estimated to give
\begin{equation*}
    \|\dA\|_{\mathcal L(\VI, \VO)} \le c_{\dA}\|\dv\|_{\Hthe,
}\end{equation*}
which proves the result.
\end{proof}

\begin{cor}\label{cor-datdp-to-zero}
    Under the assumptions of Lemma \ref{lem-da-to-zero} it holds that
    \begin{equation*}
        \dA \in \mathcal L(L^2( 0, \infty;\VI), L^2( 0, \infty; \VO))
    \end{equation*}
    and $\norm{\dA}_{\mathcal L(L^2( 0, \infty;\VI), L^2( 0, \infty; \VO))}\to 0$ as $\norm{\dv}_{\Hthe}\to 0$.
\end{cor}
% \begin{proof}
%     Standard estimates Caratheodory mapping with linear bounded operators.
% \end{proof}

\begin{rem}
    In three dimensions, the estimate \eqref{eq-esti-c} holds for $s_1 + s_2 + s_3 > \frac{3}{2}$. As the result of Lemma \ref{lem-da-to-zero} leaves no margin in the presented framework, the extension to three dimensions will either require additional regularity in $\vstst$ or lead to boundedness of $\dA$ with respect to a stronger norm than the \VI~norm.
\end{rem}

We make another assumption on the system.

\begin{ass}\label{ass-stable-output-injection}
    There exists an operator $L\in \mathcal L(\mathbb R^{k}, \VO)$ such that $A+\dA+LC$ generates an exponentially stable $C_0$-semigroup $S$, with $\|S(t) \|_{\VO}\leq Me^{-\omega t}$ , with constants $M>0$ and $\omega>0$ independent of $\dA$ for all perturbations $\dv$ with  $\norm{\dv}_{\Hthe} < \varepsilon$ for some $\varepsilon> 0$.
\end{ass}
\begin{rem}
    When considering robust output feedback stabilization it is necessary to assume that the unperturbed system is stabilizable and detectable. With the established boundedness of $B$ and the assumed boundedness of $C$, the existence of such a stabilizing output injection $L$ for the \emph{Oseen} system, follows, e.g., from the results on Riccati based controllers as provided in \cite[Ch. 6.2]{CurZ95}. As $A+LC$ is stabilized directly by the full state rather than through an observer, it has the stability margins of state feedback, cf. \cite{Doy78}. Thus stability for small perturbations $\dA$ may well be assumed and the uniformity in the constants can be achieved by worst case estimates.
\end{rem}

We can now state our main result.
\begin{thm}
    Consider the perturbed Oseen system \eqref{eq-oseen-as-lti} and let $L\in \mathcal L(\mathbb R^{k} ,\VO)$ and $\dA$ as in Assumption \ref{ass-stable-output-injection}.
    %such that $A+\dA+LC$ is uniformly exponentially stable for all $\dA$ that are defined by. 
    Then the associated transferfunction $\dG$ has a coprime factorization
    \begin{equation*}
        \dG = [N+\dN][M+\dM]^{-1},
    \end{equation*}
    where $NM^{-1}=G$ is the transferfunction associated with the unperturbed system, and
    \begin{equation*}
        \norm{\dN}_{\Hinf} \to 0 \quad\text{and}\quad \norm{\dM}_{\Hinf} \to 0
    \end{equation*}
    as $\dv \to 0$.
\end{thm}
\begin{proof}

    By Theorem \ref{thm-cpf-prtbtns}, the perturbed transfer functions can be factorized as $
    \dG = [N+\dN][M+\dM]^{-1}$ with 
    \begin{align*}
        \dN(s) = C\dA(sI-A-LC)^{-1}(sI-A-\dA-LC)^{-1}B,\\
        \dM(s) = C\dA(sI-A-LC)^{-1}(sI-A-\dA-LC)^{-1}L.
    \end{align*}
    We realize $\dN(s) = C\dA(sI-A-LC)^{-1}(sI-A-\dA-LC)^{-1}B$ as a cascaded system
    \begin{equation}\label{eq-vone-sys}
        \dot v_1 = (A + \dA + LC) v_1 + Bu
    \end{equation}
    and 
    \begin{subequations}\label{eq-vone-to-y-sys}
    \begin{align}
        \dot v &= (A + LC) v + v_1 \label{eq-vone-to-y-sys-a}\\
        y &= C\dA v
    \end{align}
\end{subequations}
and show that the overall input $u$ to output $y$ map is bounded with a constant that approaches zero as $\dv\to 0$.
    
    For system \eqref{eq-vone-sys}, from the uniform stability of $A+\dA+LC$ and the boundedness of $B$, we infer the uniform boundedness of
    \begin{equation*}
        \mathcal F_1 \colon L^2((0, \infty); \mathbb R^{2} ) \to L^2((0, \infty); \VO)\colon u \mapsto v_1,
\end{equation*}
so that 
$$\|\mathcal F_1\|_{\mathcal L (L^2((0, \infty)\mathbb R^{2});L^2((0,\infty);\VO))}<c_1$$
with a constant $c_1$ independent of $\dA$. % This regularity result follows from Corollary \ref{cor-bounded-itos} and from $\VI \hookrightarrow \VO$. 

By Corollary \ref{cor-bounded-itos}, we conclude that system \eqref{eq-vone-to-y-sys} has a bounded input to state map
\begin{equation*}
    \mathcal F_2 \in \mathcal L(L^2( (0,\infty);\VO);L^2( (0,\infty);\VI) ) \colon v_1 \mapsto v.
\end{equation*}
Note that $\mathcal F_2$ is defined by \eqref{eq-vone-to-y-sys-a} only so that its norm is independent of $\dA$.

By Corollary \ref{cor-datdp-to-zero} and by the boundedness of $C$, we have that $C\dA$ is bounded as a linear map from $\VI$ in $\mathbb R^{k}$ with a norm that approaches zero as $\dv$ goes to zero. 

Putting all together, we infer for the input to output map $\mathcal F= C\dA\mathcal F_2\mathcal F_1\colon u \mapsto y$ via the connection of Systems \eqref{eq-vone-sys} and \eqref{eq-vone-to-y-sys} that
    \begin{align}
        \norm{\mathcal F}_{\mathbb R^{k}\leftarrow \mathbb R^{2}}%^{(0,\infty)}
        &\leq 
        \|C\|_{\mathbb R^{k} \leftarrow \VO }
        \|\dA\|_{\VO \leftarrow \VI}
        \|\mathcal F_2 \|_{\VI \leftarrow \VO}
        \|\mathcal F_1\|_{\VO \leftarrow \mathbb R^{2} },
    \end{align}
    where we have used, e.g., $\VI \leftarrow \VO$ as a shortcut for $\mathcal L(L^2( (0,\infty);\VO);L^2( (0,\infty);\VI) )$.
    Thus, by Theorem \ref{thm-weiss-transferfunc}, we obtain that 
    \begin{equation*}
        \|\dN\|_{\Hinf} \le \|\mathcal F\|_{\mathcal L(L^2( (0,\infty);\mathbb R^{2} );L^2( (0,\infty);\mathbb R^{k} ) )} \to 0
\end{equation*}
as $\norm{\dv}_{\Hthe}\to 0$ and, thus, $\|\dA\|_{\mathcal L(L^2( (0,\infty);\VO);L^2( (0,\infty);\VI) )}\to 0$. 

Since $\dN$ and $\dM$ only differs in \emph{their input operator} in the corresponding first system \eqref{eq-vone-sys} and since $L$ is bounded as is $B$, the statement that $\|\dM\|_{\Hinf} \to 0$  as $\norm{\dv}_{\Hthe}\to 0$, follows analoguously.
\end{proof}

\section{Conclusion}

We have presented a general approach to qualify the linearization errors in the controller design via coprime factor uncertainties. For the \emph{Oseen} system, we have shown that the disturbance in the coprime factors smoothly depend on the disturbance in the linearization. The developed realization of the coprime factor disturbances may even be used to quantify the disturbances. Accordingly, the coprime factor uncertainties approach zero in the relevant norm when the linearization gets more accurate. Thus the model will eventually reach a region, where a robust controller based on the inexact model will stabilize the actual system. 

Although we have considered a particular setup, the arguments were general enough to be adapted to similar problems. In view of three dimensional setups, we have pointed out where the possible lack of regularity of solutions needs to be compensated.

The presented results were exclusively for the infinite dimensional system. Since it is known that discretization errors lead to similar coprime factors uncertainties, see, e.g. \cite{Cur06, BenH17}, the results may become relevant for the controller design by space-discrete models. In order to define finite dimensional controllers that can compensate both linearization and discretization errors, among others, the interplay of these error sources needs further investigation. Another urgent issue is the existence of uniformly -- with respect to the discretization -- stabilizing feedback that is commonly assumed to show convergence of control laws; see, e.g. \cite{Ito87} and the remarks in \cite{BenH17} on the difficulties that come with the nonstandard Galerkin approximations as they are used for flow equations.

Finally, it seems worth investigating how the provided results can be extended to treat Dirichlet boundary feedback control as in \cite{Ray06, Ray07a} without resorting to the Robin relaxation.

% \begin{rem}
%     we note that from $\dA$ being a closed perturbation of $A+LC$ with a domain that encloses $D(A)$ and from 
%     \begin{equation}
%         \|A+\dA+LC\|_{\VO} \leq \| A+LC \| \|v\| + \|\dA\| \|v\|
%     \end{equation}
% \end{rem}

\bibliographystyle{abbrv}
% \bibliography{../bibfiles/mor,../bibfiles/csc,../bibfiles/bib-dae-hinf,bib_jh}
% \bibliography{../bibfiles/bib-dae-hinf,bib_jh}
% \bibliography{bib_jh}
\bibliography{nse-inf-arxiv.bib}
% \bibliography{../../../../misc/bibfiles/mybibfiles/bib_jh}
\end{document}

%% file: def.tex
\DeclareMathOperator{\dive}{div}

\providecommand{\inva}[1]{\text{~\textup{d}} #1}
\providecommand{\ltp}[2]{\bigl ( #1, #2 \bigr )_{L^2}}

\providecommand{\norm}[1]{\lVert#1 \rVert}
\def\trp{{\mathsf{T}}}

\def\VO{\ensuremath{{V}^0}}
\def\VI{\ensuremath{{V}^1}}

\def\Ltod{\ensuremath{L^2(\Omega;\mathbb R^{2})}}
\def\Hood{\ensuremath{H^1(\Omega;\mathbb R^{2})}}

\def\Hthe{H^{3/2+\varepsilon_0}(\Omega; \mathbb R^{2})}

\def\Hinf{\ensuremath{H_\infty}}
\def\mhamz{\ensuremath{{\mathcal M \hat{\mathcal A}}_{-}(0)}}
\def\Czp{\ensuremath{\mathbb C_0^+}}

\def\gai{\ensuremath{\Gamma_{\text{in}}}}
\def\gao{\ensuremath{\Gamma_{\text{out}}}}
\def\gac{\ensuremath{\Gamma_1}\cup\ensuremath{\Gamma_2}}
\def\gaco{\ensuremath{{\Gamma}_1}}
\def\gact{\ensuremath{{\Gamma}_2}}
\def\gaci{\ensuremath{{\Gamma_i}}}
\def\gaw{\ensuremath{\Gamma_w}}

\def\infv{\ensuremath{v}}
\def\infp{\ensuremath{p}}

\def\vd{\ensuremath{v_\delta}}
\def\vstst{\ensuremath{v_\infty}}
\def\pstst{\ensuremath{p_\infty}}
\def\pd{\ensuremath{p_\delta}}

\def\dv{\ensuremath{\delta_v}}
\def\dA{\ensuremath{\delta_A}}
\def\dG{\ensuremath{G_\delta}}
\def\dM{\ensuremath{\delta_M}}
\def\dN{\ensuremath{\delta_N}}

%% file: mathEnv.tex
\newtheorem{thm}{Theorem}[section]
\newtheorem{cor}[thm]{Corollary}
\newtheorem{lem}[thm]{Lemma}
\newtheorem{ass}[thm]{Assumption}

\theoremstyle{remark}

\theoremstyle{remark}
\newtheorem{rem}[thm]{Remark}

\theoremstyle{definition}

\theoremstyle{plain}

\newtheoremstyle{named}{}{}{\itshape}{}{\bfseries}{.}{.5em}{#1 \thmnote{#3}}
\theoremstyle{named}

\theoremstyle{plain}

\theoremstyle{plain}

\floatstyle{ruled}
\newfloat{algorithm}{h}{loa}[section]
\floatname{algorithm}{Algorithm}

%% file: tikz/cylsetup.tex
	\pgfmathsetmacro{\domfac}{7.5}
	\pgfmathsetmacro{\domx}{1.*\domfac}
	\pgfmathsetmacro{\domy}{.4*\domfac}
	\pgfmathsetmacro{\cylrad}{.05*\domfac}
	\pgfmathsetmacro{\cylcx}{.3*\domfac}
	\pgfmathsetmacro{\cylcy}{.2*\domfac}
	\begin{tikzpicture}
		\draw (\cylcx, \cylcy) circle (\cylrad);
		\draw [white,thick,domain=30:60] plot ({\cylcx+\cylrad*cos(\x)}, {\cylcy+\cylrad*sin(\x)});
		\draw [white,thick,domain=300:330] plot ({\cylcx+\cylrad*cos(\x)}, {\cylcy+\cylrad*sin(\x)});
		\node at (\cylcx+\cylrad,{\cylcy+\cylrad*cos(30)}) [above] {${\Gamma_1}$};
		\node at (\cylcx+\cylrad,{\cylcy-\cylrad*cos(30)}) [below] {${\Gamma_2}$};
		\node at (\cylcx-\cylrad,\cylcy) [left] {$\Gamma_w$};
		\draw (0,0) -- node[below]{$\Gamma_w$}(\domx,0) ;
		\draw (\domx,\domy) -- node[above]{$\Gamma_w$}(0,\domy);
		\draw [dashed] (\domx,0) -- node[right]{$\gao$}(\domx,\domy);
		\draw [dashed] (0,\domy) -- node[right]{$\gai$}(0,0);
	\end{tikzpicture}